\documentclass[12pt,reqno]{amsart}

\usepackage{soul}
\usepackage{amsthm}
\usepackage{amssymb}
\usepackage{latexsym}
\usepackage{multicol,multirow,stmaryrd}
\usepackage{verbatim,enumerate}
\usepackage{accents,upgreek}
\usepackage[usenames]{color}
\usepackage[colorlinks=true, linkcolor=blue, citecolor=green, urlcolor=blue]{hyperref}
\usepackage[nameinlink,noabbrev,capitalize]{cleveref}
\usepackage{nicematrix,makecell}
\usepackage{amsmath, amscd,mathtools}
\usepackage{soul}
\usepackage{palatino}
\usepackage{pifont}
\usepackage{dsfont}
\usepackage{dirtytalk,caption,subcaption}
\usepackage{tikz,setspace}
\usetikzlibrary{matrix,arrows,decorations.pathmorphing}
\usetikzlibrary{positioning}
\advance\textwidth by 1.2in \advance\oddsidemargin by -.6in \advance\evensidemargin by -.6in

\theoremstyle{plain}
\newtheorem{prop}{Proposition}
\newtheorem*{thm}{Theorem}
\newtheorem{lem}{Lemma}
\newtheorem{cor}{Corollary}

\theoremstyle{definition}
\newtheorem{example}{Example}
\newtheorem{defn}{Definition}

\theoremstyle{remark}
\newtheorem{rem}{Remark}

\newcommand{\lie}[1]{\mathfrak{#1}}

\newcommand\bc{\mathbb C}
\newcommand\bz{\mathbb Z}
\newcommand\bn{\mathbb N}

\title{$\pi$-systems and the embedding problem for rank $2$ Kac-Moody Lie algebras}

\author{Irfan Habib}\address{Department of Mathematics, Indian Institute of Science, Bangalore 560012}
\email{irfanhabib@iisc.ac.in}
\thanks{I.H. was partially supported by DST/INSPIRE/03/2019/000172.}
\author{Chaithra P}
\address{Department of Mathematics, Indian Institute of Science, Bangalore 560012}
\email{chaithrap@iisc.ac.in}
\thanks{C.P. was partially supported by Prime Minister Research Fellowship.}

\begin{document}
\begin{abstract}
    $\pi$-systems are fundamental in the study of Kac-Moody Lie algebras since they arise naturally in the embedding problems. Dynkin introduced them first and showed how they also appear in the classification of semisimple subalgebras of a semisimple Lie algebra. In this article, we explicitly classify the $\pi$-systems associated to rank $2$ Kac-Moody Lie algebras and prove that in most of the cases they are linearly independent. This classification allows us to determine the root generated subalgebras and which in turn determines all possible Kac-Moody algebras that can be embedded in a rank $2$ Kac-Moody algebra as subalgebras generated by real root vectors. Additionally, following the work of Naito we provide examples illustrating how Borcherds Kac-Moody algebras can also be embedded inside a rank $2$ Kac-Moody algebra.
\end{abstract}
\maketitle
\section{Introduction}

In one of his influential papers, Dynkin classified semisimple subalgebras of finite dimensional semisimple Lie algebras (see \cite{dynkin1952semisimple}). An important key ingredient to achieve this goal was the study and classification of so-called symmetric regular subalgebras; these are, by definition, invariant under the action of a fixed Cartan subalgebra and the corresponding set of roots is symmetric. Furthermore, for a given finite-dimensional semisimple Lie algebra $\mathring{\lie g}$ with set of roots $\mathring{\Delta}$, Dynkin obtained bijective correspondences between the following sets: 
\begin{enumerate}
 \item  the set of symmetric regular subalgebras of $\mathring{\lie g}$,
     \item the set of (linearly independent) $ \pi\text{- systems of }\mathring{\Delta}\ \text{contained in  } \mathring{\Delta}^+$,
     \item the set of $ \text{closed subroot systems of } \mathring{\Delta}$.
 \end{enumerate}
 Using this correspondence, Dynkin reduced the problem of classifying symmetric regular subalgebras to the combinatorial problem of classifying closed subroot systems. The classification of closed subroot systems can be obtained from the classification of maximal closed subroot systems, which has been achieved in \cite{BdS}. Moreover, any symmetric regular subalgebra corresponds to a Cartan-invariant semisimple subalgebra of $\mathring{\lie g}$. This, in principle, finishes the classification of regular semisimple subalgebras of a finite dimensional semisimple Lie algebra. 
 
 The analogue picture for an arbitrary Kac-Moody algebra $\lie g$ fails and first counterexamples can be constructed even for low rank untwisted affine Lie algebras. Although the structure of symmetric regular subalgebras is not completely determined by the combinatorics of their root systems, the first step in understanding them is to investigate the combinatorial picture. For symmetrizable Kac-Moody algebras, it is a tough ask to classify all the Cartan invariant subalgebras. Almost after $70$ years of Dynkin's work, Roy and Venkatesh classified the closed subroot systems of an affine root system and proved that similar results to Dynkin hold for affine Lie algebras (see \cite[Theorem 11.3.1]{roy2019maximal} for a precise statement). This discovery is very instigating and prompts us to ask whether similar phenomena holds for a general Kac-Moody Lie algebra. Before discussing about the existing literature for the general Kac-Moody case, we point out that there are attempts to understand the regular subalgebras of an affine Kac-Moody algebra both from an algebraic and combinatorial point of view. In \cite{barnea1998graded} the authors classified the maximal graded subalgebras of an affine Kac-Moody algebra. Anna Felikson et al. started the classification of regular subalgebras of affine Kac-Moody algebras (which are related to the closed subroot systems of real affine root systems) in \cite{felikson2008regular} combinatorially and it was completed in \cite{roy2019maximal}, see also \cite{KV21a}. They also provide the complete list of maximal rank root subroot systems of an affine root system \cite[Table 1]{felikson2008regular}.

 For a symmetrizable Kac-Moody algebra, the key initial results about $\pi$-systems and regular subalgebras were obtained by Morita in \cite{morita1989certain} and Naito in \cite{naito1992regular}. Morita classified certain rank $2$ subsystems of a symmetrizable Kac-Moody algebra using linearly independent $\pi$-systems. Naito, on the other hand, generalized the definition including imaginary roots in a $\pi$-system. With this extended definition, he showed that even algebras of Borcherds Kac-Moody type can be embedded inside a Kac-Moody algebra. We demonstrate this fact in Section \ref{borcherds} by examples. Feingold and Nicolai rediscovered the definition of a $\pi$-system to study subalgebras of Hyperbolic Kac-Moody algebra in \cite{feingold2004subalgebras}. They relaxed the condition of linear independence but imposed the restriction that a $\pi$-system contains positive real roots only. In the same paper, using the $\pi$-system technique, they produced several families of indefinite type subalgebras of a hyperbolic Kac-Moody algebra. Carbone et al. studied the embedding problem in \cite{carbone2021varvec} and they provided a simpler proof of the following Theorem which is originally from Naito. This Theorem shows how $\pi$-systems are closely related to the embedding problem.
 \begin{thm}[\cite{carbone2021varvec,naito1992regular}]
     Let $A$ be a finite GCM and $\Sigma$ be a linearly independent $\pi$-system in $\lie \lie g(A).$ Let $B=B_\Sigma$ be the GCM corresponding to $\Sigma$ and  $\lie g(B)$ be the Kac-Moody algebra. Then there exists an embedding of Lie algebras $\lie g(B) \hookrightarrow \lie g(A).$
 \end{thm}
 Very recently, Irfan et al. in \cite{habib2023root} proved that Dynkin's results hold in the symmetrizable Kac-Moody setting. In particular, they proved that the classification of root generated subalgebras is equivalent to the classification of $\pi$-systems consisting of real positive roots. In general, for a closed subroot system $\Psi$ (see Definition \ref{keydefn}), the $\pi$-system $\Pi(\Psi)$ associated with $\Psi$ is neither finite nor linearly independent. Even if $\Pi(\Psi)$ is finite but not linearly independent, the algebra $\lie g'(B_{\Pi(\Psi)})$ (see Section \ref{secpi}) has more relations than the usual relations plus the Serre relations as pointed out in \cite[Page 14, Comment 1]{Henn06Geometric}. This makes the study of the root generated subalgebras quite challenging. In \cite{habib2023root} the authors identified the extra relations by determining the kernel of the map $\lie g'(B_{\Pi(\Psi)})\to \lie g(A).$ They proved that the kernel of the abovementioned map is contained in the centre of $\lie g'(B_{\Pi(\Psi)})$ and thus no essential information about the roots is lost. It seems that \cite{habib2023root} is the first article in the literature that explicitly deals with infinite $\pi$-systems in the context of the embedding problem.

 Nevertheless, the study of $\pi$-systems holds significance in its own right. They are widely used in physics to understand the regular embedding problem; see \cite{damour03cosmological, feingold83ahyperbolic, superspace81} for an exposition.
 The real roots and the representations of $E_{10}$ have been shown to correspond to the fields of $11$ dimensional supergravity at low levels \cite{damour03cosmological}. Viswanath proved that the Lie algebra $E_{10}$ contains a regular subalgebra isomorphic to $AE_3$ \cite{viswanath08embedding}. This reflects the inclusion of Einstein gravity into $11$-dimensional supergravity. Later in \cite[Theorem 6.2]{carbone2021varvec} the authors proved that all $\pi$-systems of type $AE_3$ in $E_{10}$ are conjugate (upto negatives) under the Weyl group of $E_{10},$ indicating that there is a canonical inclusion of Einstein gravity into $11$- dimensional supergravity. We also remark that there is another way to embed $AE_3$ into $E_{10}$ via the gravity truncation method of \cite{damour09fermionic}. In \cite{Henn06Geometric} the authors studied the solutions with non-zero electric field determined by geometric configurations $(n_m,g_3), n\le 10$ and they showed that these solutions are associated to rank $g$ regular subalgebras of $E_{10}.$

 The lack of detailed information about $\pi$-systems and regular subalgebras beyond the affine case motivates us to study the $\pi$-systems of rank $2$ Kac-Moody Lie algebra $\lie g$. We shall use the explicit description of real roots given in \cite{kac1990infinite}. To understand the symmetric subalgebras generated by real root vectors of $\lie g,$ it is enough to understand the $\pi$-systems contained in the positive real roots (c.f. \cite[Theorem 1]{habib2023root}). The complete classification of such $\pi$-systems is achieved in Theorem \ref{Thmab}. In \cite[Table 2]{carbone2021commutator} the authors listed the GCM for all rank $2$ subroot systems of $\lie g.$ Along with \cite[Theorem 1]{habib2023root} their classification describes the possible Cartan types of the root generated subalgebras of $\lie g.$ Our classification aids in determining the structure and Cartan types of these subalgebras in a more explicit fashion by determining the Cartan matrix of a given root generated subalgebra, as summarized in Table \ref{tablesummary}. Additionally, we go one step further and completely classify the $\pi$-systems in rank $2$, including those containing negative roots. The complete classification of $\pi$-systems has been obtained in Theorem \ref{NThmab} by determining precisely when a sum of two real roots is a root (real or imaginary). It is worth noting that the article \cite{carbone2021commutator} has already determined when such sums give rise to real roots. Finally, we extend our analysis to include the non-reduced affine case in the appendix.\medskip

 \noindent \textit{The paper is organised as follows:} Let $\lie g$ be a rank $2$ Kac-Moody algebra. In Section \ref{secprelim} we recall the necessary definitions and properties of Kac-Moody algebras, especially for rank $2$ Kac-Moody algebras. We also explain how $\pi$-systems appear naturally in the embedding problem. In Section \ref{lipisystems} we classify the linearly independent $\pi$-systems of $\lie g$ and describe the root generated subalgebras upto isomorphism. In Section \ref{clpisystems} we classify all the $\pi$-systems of $\lie g.$ In Appendix \ref{app} we describe the $\pi$-systems of the non-reduced affine Lie algebra of rank $2.$

 This research was greatly facilitated by experiments carried out in the computational algebra systems SageMath \cite{sagemath}.


\section{Preliminaries}\label{secprelim}
\subsection{Notations} Throughout this paper we denote by $\mathbb{C}$ the field of complex numbers and by $\mathbb{Z}$ (resp. $\mathbb{Z}_{+}$, $\mathbb{N}$, $\mathbb{R}$) the set of integers (resp. non-negative, positive integers, real numbers). 

\subsection{Kac-Moody algebras} Let $A=(a_{ij})$ be a symmetrizable generalized Cartan matrix (GCM in short) of order $n$ and let $\lie g(A)$ be the corresponding Kac-Moody Lie algebra with Cartan subalgebra $\lie h(A)$. Then $\lie g(A)$ is generated as a Lie algebra by $\lie h(A),\ e_i^{\pm}:=e_i(A)^\pm,\ i=1,2,\cdots,n$ satisfying the relations
$$[h,h']=0,\ [h,e_i^\pm]=\pm\alpha_i(h)e_i^\pm,\ [e_i^+,e_j^-]=\delta_{ij}\alpha_i^\vee,\ \ h,h'\in\lie h(A),\ i=1,2,\cdots,n,$$
$$(\mathrm{ad}e_{i}^\pm)^{1-a_{ij}}(e_j^\pm)=0,\ \ 1\le i\neq j\le n,$$ where $\{\alpha_1^\vee,\alpha_2^\vee,\cdots,\alpha_n^\vee\}$ and $\{\alpha_1,\alpha_2,\cdots,\alpha_n\}$ are linearly independent subsets of $\lie h(A)$ and $\lie h(A)^*$ respectively such that $\langle \alpha_j,\alpha_i^\vee\rangle=a_{ij}.$ For $\lambda\in\lie h(A)^*$ and $h\in \lie h(A),$ we denote the usual pairing by $\langle \lambda,h\rangle.$

Let $\Delta(A)$ (resp. $\Delta_{\mathrm{re}}(A), \Delta_{\mathrm{im}}(A)$) be the set of all roots (resp. real roots, imaginary roots) of $\lie g(A)$ with respect to $\lie h(A).$ Let $(\cdot,\cdot)_{A}$ be the symmetric, invariant bilinear form on $\lie g.$ Then we have $$\Delta_{\mathrm{re}}(A)=\{\alpha\in\Delta:(\alpha,\alpha)_A>0\},\ \ \ \Delta_{\mathrm{im}}(A)=\{\alpha\in\Delta:(\alpha,\alpha)_A\le 0\}.$$ We choose a Borel subalgebra $\lie b(A)$ of $\lie g(A)$ and denote by $\Delta^+(A)$ (resp. $\Delta^-(A)$) the set of all positive (resp. negative) roots of $\lie g(A).$ We also define $$\Delta_{\mathrm{re}}^{\pm}(A):=\Delta^{\pm}(A)\cap \Delta_{\mathrm{re}}(A),\ \ \ \Delta_{\mathrm{im}}^{\pm}(A):=\Delta^{\pm}(A)\cap \Delta_{\mathrm{im}}(A).$$
Let $W(A)$ be the Weyl group of $\lie g(A)$ which is a Coxeter group generated by reflections $s_\alpha,\alpha\in\Delta_{\mathrm{re}}(A).$ For a subset $S\subseteq\Delta_{\mathrm{re}}(A),$ set $W_S(S):=\langle s_\beta:\beta\in S \rangle.$
Let $\{\alpha_1(A),\alpha_2(A),\cdots,\alpha_n(A)\}$ be the set of all simple roots of $\lie g(A).$ We denote by $Q(A)$ (resp. $Q^\pm(A)$) the set of all $\bz$ (resp. $\bz_{\pm}$) linear combination of $\alpha_1(A),\alpha_2(A),\cdots,\alpha_n(A).$ Most of the times, we shall suppress the matrix $A$ and write $\lie g,\Delta,\Delta^+$ etc for $\lie g(A),\Delta(A),\Delta^+(A)$ etc when $A$ is clear from the context. In general, determining which elements of $Q$ is a root is a tough ask but we have a good description of roots when $A$ is of finite, affine or hyperbolic type (\cite[Proposition 5.10]{kac1990infinite}).
\begin{lem}\label{keylem}
    Let $A$ be a GCM of finite, affine or hyperbolic type. Then 
    \begin{enumerate}
        \item the set of all real roots of $\lie g$ is given by $$\left\{\alpha=\sum_{i=1}^n k_i \alpha_i\in Q:(\alpha,\alpha)>0\text{ and } \frac{k_j(\alpha_j,\alpha_j)}{(\alpha,\alpha)}\in \bz\text{ for all }1\le j\le n\right\},$$
        \item and the set of all imaginary roots of $\lie g$ is given by 
        $\{\alpha\in Q\backslash \{0\}: (\alpha,\alpha)\le 0\}.$
    \end{enumerate}
\end{lem}
\begin{defn}
    A $\lie h$-invariant Lie subalgebra of $\lie{g}$ is called a \textit{regular subalgebra}. Given a regular subalgebra $\lie{s}$, we define
$$\Delta(\lie{s}):=\{\alpha\in\Delta:(\lie{g}_{\alpha}\cap \lie s)\neq 0\},\ \ \Delta(\lie{s})_{\mathrm{re}}:=\Delta(\lie{s})\cap\Delta_{\mathrm{re}},\ \ \lie{s}_{\alpha}:=\lie{g}_\alpha\cap \lie{s},\ \alpha\in\Delta(\lie s)$$ to be the set of roots of $\lie s$ with respect to $\lie h$.  We call $\lie s$ \textit{symmetric} if $\Delta(\lie{s})=-\Delta(\lie{s})$.
\end{defn}
Given a subset $\Psi\subseteq\Delta,$ we define $\lie g(\Psi)$ to be the subalgebra of $\lie g$ generated by $x_{\pm\alpha},\ \alpha\in\Psi,$ where $x_{\pm\alpha}\in\lie g_{\pm\alpha}$ is any non-zero root vector. Note that $\lie g(\Psi)$ is always a regular subalgebra of $\lie g.$ A regular subalgebra $\lie s$ is called a \textit{root generated subalgebra} if $\lie s=\lie g(\Psi)$ for some $\Psi\subseteq \Delta_{\mathrm{re}}$ with $\Psi=-\Psi.$ The roots of regular subalgebras satisfy certain combinatorial properties that we encode in the following definition.


\begin{defn}\label{keydefn}
A non-empty subset $\Psi\subseteq \Delta$ is called 
\begin{itemize}

    
    \item \textit{ a subroot system} if $s_{\alpha}(\beta)\in \Psi$ for all $\alpha\in \Psi\cap \Delta^{\mathrm{re}}$ and $\beta\in \Psi$.
    
    \vspace{0,2cm}
    

       
    \item \textit{real closed} or closed if $\Psi\subseteq \Delta^{\mathrm{re}}$ and the condition holds
    $$\alpha+\beta\in \Delta^{\mathrm{re}},\ \alpha,\beta\in \Psi \implies \alpha+\beta\in \Psi.$$
    
    \item \textit{closed subroot system} if it is real closed and a subroot system.
\end{itemize}
\end{defn}

When $\lie g$ is finite dimensional, Dynkin \cite{dynkin1952semisimple} proved that semisimple subalgebras of $\lie g$ are in one to one correspondence to the closed subroot systems of $\Delta.$ Moreover, it is well known that there is a one to one correspondence between closed subsets of $\Delta$ and the regular subalgebras of $\lie g$ (see \cite[Proposition 4.1]{DdeG21} for the precise statement). Almost $70$ years later, Roy et al. proved that similar result to Dynkin holds when $\lie g$ is of affine type in \cite{roy2019maximal}. Very recently, the authors extended the results for all symmetrizable Kac-Moody algebras in \cite{habib2023root}.

In a first step, it is natural to attack the problem of understanding the structure of regular subalgebras in a purely combinatorial way. The classification of subsets defined above is quite challenging and wide open for Kac-Moody algebras. In particular cases, for example when $\Delta$ is a finite root system, $\Psi$ is a maximal real closed subroot system or a symmetric real closed subset of the roots of an (extended) affine Lie algebra, classifications are given in the literature (see for example \cite{BHV23a,DdeG21,KV21a,roy2019maximal} and references therein).
For future reference, we end this subsection with the following Lemma which can be found in \cite{habib2023root}.
\begin{lem}\label{sumnotreal}
    If $a_{ij}\le -2$ for all $i\neq j,$ then $\alpha+\beta\notin \Delta^{\mathrm{re}}$ for any $\alpha,\beta\in \Delta^{\mathrm{re}}.$ 
    In particular, any subset $\Psi\subseteq \Delta^{\mathrm{re}}$ is real closed. 
    \end{lem}

\subsection{Rank \texorpdfstring{$2$}{2} Kac-Moody algebras} Let $A=\begin{psmallmatrix}
    2& -a\\ -b & 2
\end{psmallmatrix},\ a\geq b$ is a GCM of order $2.$ The matrix 
$A$ is of finite (resp. affine) type if and only if $\det(A)>0$ (resp. $\det(A)=0$) and hyperbolic type otherwise. The positive real roots of $\lie g=\lie g(A)$ are given by \cite{kac1990infinite}
\begin{equation}\label{realrootsab}
    \Delta^{re}_+=\{c_j\alpha_1+d_{j+1}\alpha_2,\ c_{j+1}\alpha_1+d_j\alpha_2:j\in\bz_+\}.
\end{equation}

where $c_j,d_j$ are non-negative integers given by the recurrence relations:
\begin{equation}\label{recurrenceab}
  c_{k+2}+c_k=ad_{k+1},\ \ d_{k+2}+d_k=bc_{k+1},\ \ c_0=d_0=0,\ c_1=d_1=1.
\end{equation}
The next Lemma is easy to prove by induction. 
\begin{lem}\label{lemrelcd}
    For all $j\in \bz_+$ we have $c_{2j+1}=d_{2j+1}$ and $c_{2j}=\frac{a}{b}d_{2j}.$
\end{lem}
\noindent The symmetric invariant bilinear form on $\lie h^*$ is given by 
$$(\alpha_1,\alpha_1)=2,\ \ \ (\alpha_2,\alpha_2)=\frac{2a}{b},\ \ \ (\alpha_1,\alpha_2)=-a.$$
For convenience, we define 
\begin{equation}\label{realrootstypes}
    \beta_{1}^{j}=c_j\alpha_1+d_{j+1}\alpha_2,\ \ \ \beta_{2}^{j}=c_{j+1}\alpha_1+d_{j}\alpha_2,\ \ \ j\in\bz_+.
\end{equation} We shall call a root of the form $\beta_{1}^{j}$ (resp. $\beta_{2}^{j}$) a root of type $I$ (resp. type $II$).\medskip

\textit{Throughout this article for a rank $2$ GCM $A=\begin{psmallmatrix}2& -a\\ -b & 2\end{psmallmatrix},$ we shall assume that $a\geq b.$}

\subsection{\texorpdfstring{$\pi$}{pi}-systems and embedding problem}\label{secpi}

In this subsection we shall recall the importance of $\pi$-systems in the embedding problem for Kac-Moody Lie algebras. We begin with the following definition.
\begin{defn}
    A subset $\Sigma\subseteq \Delta_{\mathrm{re}}$ is called a $\pi$-system of $\Delta$ (or simply a $\pi$-system) if $\alpha-\beta\notin \Delta$ for all $\alpha,\beta\in \Sigma.$
\end{defn}
\begin{example}
    The set of simple roots of a Kac-Moody algebra is a $\pi$-system. If $\lie g$ is a finite dimensional simple Lie algebra and $\Sigma_0$ is a set of simple roots of $\lie g,$ then $\Sigma:=\Sigma_0\cup\{-\theta\}$ is a $\pi$-system, where $\theta$ is the highest root of $\lie g$ with respect to $\Sigma_0.$ 
\end{example}

We emphasize that there is no uniform definition of a $\pi$-system in the literature. In \cite{dynkin1952semisimple} a $\pi$-system is required to be linearly independent and in \cite{feingold2004subalgebras} a $\pi$-system does not need to be linearly independent, however should contain only positive roots. The authors of \cite{carbone2021varvec} drop both restrictions. 

Given a finite $\pi$-system $\Sigma=\{\beta_1,\beta_2,\cdots,\beta_\ell\},$ we can define a matrix $B_\Sigma=(b_{ij})$ where $b_{ij}=\langle \beta_j,\beta_i^\vee \rangle.$ It  is well known that $B_\Sigma$ is a GCM (see e.g. \cite{carbone2021varvec}). In the second example above, the matrix $B_\Sigma$ is the matrix of the untwisted affine Lie algebra associated to $\lie g$. The following Lemma shows how $\pi$-systems naturally arise in the embedding problem for Kac-Moody algebras. The proof is easy and we supply the details for completeness.
\begin{lem}
    Let $A$ and $B$ be GCMs of order $n$ and $m$ respectively. Assume that there is an injective Lie algebra homomorphism $\iota:\lie g(B)\hookrightarrow \lie g(A)$ such that the $\iota(e_i^\pm(B))$ are real root vectors in $\lie g(A)$ for each $i=1,2,\cdots,m.$ Then $B=B_\Sigma$ for a unique $\pi$-system $\Sigma\subseteq \Delta^+_{\mathrm{re}}(A).$
\end{lem}
\begin{proof}
    Let $\iota(\lie g(B))=\lie s$ and let $\{\beta_1,\beta_2,\cdots,\beta_m\}\subseteq \Delta_{\mathrm{re}}$ be such that $\iota(e_i^+(B))\in \lie g(A)_{\beta_i},\ i=1,2,\cdots,m.$  It is clear that $\lie s=\lie g(A)(\Psi),$ where $\Psi$ is the real closed subroot system of $\Delta(A)$ generated by $\{\beta_i:i=1,2,\cdots,m\}$. The proof is now completed by \cite[Theorem 1]{habib2023root} (the next Theorem).
\end{proof}

The following Theorem extends the results in \cite{dynkin1952semisimple,roy2019maximal} and serves as a bridge between the algebraic and combinatorial side for symmetrizable Kac-Moody algebras. 

\begin{thm}\label{mainresbij}\cite{habib2023root}
     Let $\lie g$ be a symmetrizable Kac-Moody algebra.
     We have the following bijections \vspace{10pt}
    $$\begin{array}{cccccr}
     \vspace{10pt}\left\{\begin{array}{c}
            \pi-\text{systems of }\Delta  \\
             \text{contained in } \Delta^+_{\mathrm{re}}
       \end{array}\right\}  & \longleftrightarrow & \left\{\begin{array}{c}
            \text{real closed subroot}  \\
             \text{systems of } \Delta
       \end{array}\right\}& \longleftrightarrow & \left\{\begin{array}{c}
            \text{root generated}\\
             \text{subalgebras of $\lie g$} 
       \end{array}\right\} \\ \vspace{10pt}
        \Sigma  & \longmapsto & W_\Sigma(\Sigma)\\
    \Pi(\Psi) & \longmapsfrom &\Psi & \longmapsto& \lie g(\Psi)&\\ \\
    &&\Delta(\lie g(S))_{\mathrm{re}}& \longmapsfrom &\lie g(S)&
    \end{array}$$
    \end{thm}
    
    In the same article, the authors proved that for a closed subroot system $\Psi$ of $\Delta,$ we have $\lie g(\Psi)=\lie g(\Pi(\Psi)).$ It is natural to ask what kind of structure does the subalgebra $\lie g(\Psi)$ posses? It has been proved in \cite{naito1992regular} (see also \cite{carbone2021varvec}) that if $\Pi(\Psi)$ is linearly independent, then $\lie g(\Psi)$ is indeed isomorphic to the derived algebra of some Kac-Moody algebra. There are examples where $\Pi(\Psi)$ is neither finite nor linearly independent. In this article, we shall prove that for a rank $2$ Kac-Moody algebra $\lie g,$ any $\pi$-system $\Sigma\subseteq \Delta^+_{\mathrm{re}}$ is linearly independent. The following Lemma holds for any rank and for completeness, we supply the details. 
\begin{lem}\label{finpisystem}
    If $\lie g$ is of finite type, then any $\pi$-system $\Sigma\subseteq \Delta^+$ of $\lie g$ is linearly independent.
\end{lem}
\begin{proof}
    Let $\sum_{\alpha\in \Sigma} c_\alpha \alpha=0.$ If $\Sigma$ is not linearly independent, then there exists $\alpha_1,\alpha_2\in \Sigma$ such that $c_{\alpha_1}>0$ and $c_{\alpha_2}<0.$ Separating the indices for which $c_\alpha>0$ from those for which $c_\alpha<0,$ we can rewrite this as $\sum t_\alpha \alpha=\sum s_\beta \beta$ where $t_\alpha,s_\beta$ are all non-negative, each  $\alpha,\beta$ occurring in the above expression are positive roots and the set of $\alpha'$s and $\beta'$s are disjoint. Set $\gamma=\sum t_\alpha \alpha.$ Then $(\gamma,\gamma)=\sum_{\alpha,\beta} t_\alpha s_\beta(\alpha,\beta)\le 0$ and thus $\gamma=0.$ In particular, we have $t_\alpha=s_\beta=0$ for all $\alpha,\beta.$
\end{proof}

Although there is an explicit description of real roots of a rank $2$ Kac-Moody algebra, straightforward induction is not the correct tool to use to prove the results. Instead, we will employ the following simple yet highly effective version of the induction method for the majority of the proofs.
\begin{lem}\label{lemstar}
    Let $A$ be a subset of $\bz_+$ such that 
\begin{equation}\label{propstar}
\begin{split}
    &\ \ \text{there exists }k\in \bz_+\ \text{such that }k,k+1\in A,\\
    &\ \ j\in A\Longrightarrow j+2\in A\text{ for any }j\geq k\in \bz_+.
\end{split}\tag{$\star$}
\end{equation}
then we have $A\supseteq \{k,k+1,k+2,\cdots\}.$
\end{lem}

\section{Classifications of linearly independent \texorpdfstring{$\pi$}{pi}-systems}\label{lipisystems}
In this section we shall classify the linearly independent $\pi$-systems of rank $2$ Kac-Moody Lie algebras. Throughout the section we assume that $A=\begin{psmallmatrix}2 & -a \\ -b & 2\end{psmallmatrix}$ so that $a,b\in\bn$ and $a\geq b.$ Let $\lie g=\lie g(A)$ be the corresponding Kac-Moody Lie algebra.  The main result of this Section is the following Theorem.
\begin{thm}\label{Thmab}
    Let $A= \begin{psmallmatrix} 2 & -a \\-b & 2 \end{psmallmatrix}$ be a generalized Cartan matrix such that either $b\geq 2$ or $b=1$ and $a\geq 5$. Then every $\pi$-system $\Sigma\subseteq \Delta^+_{\mathrm{re}}$ of $\lie g(A)$ satisfies $|\Sigma|\le 2.$ If $|\Sigma|=2,$ then $\Sigma$ is of the form $\Sigma=\{\beta_1^j,\beta_2^k\}$ for some $j,k\in\bz_+.$ In particular, $\Sigma$ is linearly independent. Conversely,
    \begin{enumerate}
        \item if $b\geq 2,$ then $\Sigma=\{\beta_{1}^{j},\beta_{2}^{k}\}$ is a $\pi$-system for any $j,k\in\bz_+.$
        \item if $b=1$ and $a\geq 5,$ then $\Sigma=\{\beta_{1}^{j},\beta_{2}^{k}\}$ is a $\pi$-system if and only if $(j,k)\neq (1,0).$
    \end{enumerate}
\end{thm}\qed

We point out that along with \cite[Theorem 1]{habib2023root}, the statement $|\Sigma|\le 2$ for any $\pi$-system $\Sigma\subseteq\Delta^+_{\mathrm{re}}$ has been previously observed in \cite[Theorem 4.8]{carbone2021commutator}. Our aim is to classify explicitly the $\pi$-systems and use the classification to obtain all the root generated subalgebras of $\lie g$ upto isomorphism along with their associated GCM. The fact that $\Sigma$ is linearly independent holds if $A$ is of finite type by Lemma \ref{finpisystem}. Hence we can assume that $A$ is not of finite type. \medskip

\textit{Till Section \ref{clpisystems}, we shall assume either $b\geq 2$ or $b=1$ and $a\geq 5$. We postpone the remaining case $b=1$ and $a=4$ in the appendix. Note that we always have $ab\geq 4.$}\medskip

Using \cite[Theorem 3.1]{carbone2021varvec} we obtain the classification of the linearly independent $\pi$-systems as a corollary of the above Theorem.
\begin{cor}
    Let $A= \begin{psmallmatrix} 2 & -a \\-b & 2 \end{psmallmatrix}$ be a generalized Cartan matrix and $\lie g(A)$ be the corresponding Kac-Moody algebra. Let $\Sigma$ be a linearly independent $\pi$-system in $\Delta.$ Then either $|\Sigma|=1$ or there exists $w\in W$ and $\Sigma_0,$ a $\pi$-system defined by Theorem \ref{Thmab}, such that either $w\Sigma=\Sigma_0$ or $w\Sigma=-\Sigma_0.$ 
\end{cor}
Theorem \ref{Thmab} also determines the structure of root generated subalgebras of $\lie g(A),$ see \cite{carbone2021varvec,habib2023root} for their importance.
\begin{cor}
    If $\Psi$ is a closed subroot system of $\Delta,$ then the root generated subalgebra $\lie g(\Psi)$ is isomorphic to the derived subalgebra of a rank $2$ Kac-Moody Lie algebra.
\end{cor}

The rest of this Section is devoted to prove Theorem \ref{Thmab}. We point out that there is another way to obtain the positive real roots of $\lie g(A)$ by applying the Weyl group elements to the simple roots $\{\alpha_1,\alpha_2\}$ and they are related by  
\begin{equation}\label{rootfourtypes}
    \begin{split}
        \beta_{1}^{2k}=(s_2s_1)^k(\alpha_2),\ \ & \ \beta_{1}^{2k+1}=(s_2s_1)^ks_2(\alpha_1),\\
        \beta_{2}^{2k}=(s_1s_2)^k(\alpha_1),\ \ & \ \beta_{2}^{2k+1}=(s_1s_2)^ks_1(\alpha_2).
    \end{split}
\end{equation}

The following Lemma can be found in \cite{anderson10KacMoody, carbone2015root}, see also \cite[Lemma 3.1,3.5]{carbone2021commutator} for more details.
\begin{lem}\label{increasing}
     The sequences $c_j,\ d_j,\ j\in\bz_+$ satisfy
     \begin{enumerate}
         \item $c_{2j}<c_{2j+1}<c_{2j+2}$ and $d_{2j}<d_{2j+1}<d_{2j+2}$ for all $j\in\bz_+$ if $b\geq 2.$
         \item $c_{2j+1}<c_{2j}<c_{2j+3}<c_{2j+2}$ for all $j\geq 1$ and $d_{2j}<d_{2j-1}<d_{2j+2}<d_{2j+1}$ for all $j\geq 2$ if $b=1$ and $a\geq 5.$
     \end{enumerate}
     Considering the values of $\{c_i,d_i:i=0,1\}$ and $d_2,$ we obtain that all following the sequences $c_{2j},c_{2j+1},d_{2j},d_{2j+1},\ j\in\bz_+$ are strictly increasing.
\end{lem}

\subsection{} Note that from the recurrence relations \ref{recurrenceab} we obtain
\begin{equation}\label{auxeq}
    c_{k+3}=(ab-1)c_{k+1}-ad_k,\ \ \ d_{k+3}=(ab-1)d_{k+1}-bc_k.
\end{equation}
The next couple of Lemmas provide essential inequalities to prove Theorem \ref{Thmab}.
\begin{lem}\label{lem1}
    The inequalities 
    $$bc_k^2+ad_{k+1}^2-abc_kd_{k+1}+abc_{k}-2ad_{k+1}+a\le 0,\ \text{ and }$$
    $$bc_{k+1}^2+ad_k^2-abd_kc_{k+1}-2bc_{k+1}+abd_k+b\leq 0,$$ hold  
    \begin{enumerate}
        \item for all $k\geq 1$ if $b\geq 2,$
        \item for all $k\geq 2$ if $b=1$ and $a\geq 5.$
    \end{enumerate}
\end{lem}
\begin{proof}
    For $k=2$ the above expressions become $a(4-ab)$ and $b(4-ab)$ respectively and for $k=3$, they both become $a+b-ab(ab-3)$. It is easy to see that they are both non-positive integers by the given conditions. Now assume that both the equations hold for some $k\geq 2.$ Using the recurrence \ref{recurrenceab} and relations in Equation \ref{auxeq} we have 
    \begin{align*}
        &bc_{k+2}^2+ad_{k+3}^2-abc_{k+2}d_{k+3}+abc_{k+2}-2ad_{k+3}+a\\
        =& \underbrace{bc_k^2+ad_{k+1}^2-abc_kd_{k+1}+abc_{k}-2ad_{k+1}+a}+a(4-ab)d_{k+1}
    \end{align*}
    and 
    \begin{align*}
    & bc_{k+3}^2+ad_{k+2}^2-abd_kc_{k+3}-2bc_{k+3}+abd_{k+2}+b\\=&\underbrace{bc_{k+1}^2+ad_k^2-abd_kc_{k+1}-2bc_{k+1}+abd_k+b}+b(4-ab)c_{k+1}
    \end{align*}

    Both the bracketed terms are $\le 0$ by induction hypothesis and the remaining term is also non-positive since $ab\geq 4.$ Thus the set of integers for which both the equations hold, satisfies the property (\ref{propstar}) of Lemma \ref{lemstar} for $k=2$ and therefore both hold for $k\geq 2.$ Now it is easy to check that both of the inequalities are satisfied for $k=1$ if $b\geq 2.$ This completes the proof.
\end{proof}
\begin{lem}\label{lem2}
    For all $k\geq 0$ the following strict inequalities hold.
    \begin{equation}\label{funnyeq00}
        bc_{k+1}^2+ad_{k}^2-abc_{k+1}d_{k}+abc_{k+1}-2ad_{k}+a> 0,
    \end{equation}
    \begin{equation}\label{funnyeq01}
        bc_k^2+ad_{k+1}^2-abc_kd_{k+1}-2bc_k+abd_{k+1}+b> 0.
    \end{equation}
\end{lem}
\begin{proof}
    For $k=0$ both the equations become $a+b+ab$ and  for $k=1,$ they become $a^2b$ and $ab^2$ respectively and clearly they are positive integers. Now using the recurrence \ref{recurrenceab} and relations in Equation \ref{auxeq} we have 
    \begin{align*}
        & bc_{k+3}^2+ad_{k+2}^2-abc_{k+3}d_{k+2}+abc_{k+3}-2ad_{k+2}+a\\
        &=bc_{k+1}^2+ad_{k}^2-abc_{k+1}d_{k}+abc_{k+1}-2ad_{k}+a+ab(ab-4)c_{k+1}+a(4-ab)d_{k}\\
        &=\underbrace{bc_{k+1}^2+ad_{k}^2-abc_{k+1}d_{k}+abc_{k+1}-2ad_{k}+a}+a(ab-4)d_{k+2}
    \end{align*}
     and 
     \begin{align*}
        &bc_{k+2}^2+ad_{k+3}^2-abc_{k+2}d_{k+3}-2bc_{k+2}+abd_{k+3}+b\\
        &=bc_k^2+ad_{k+1}^2-abc_kd_{k+1}-2bc_k+abd_{k+1}+b+b(4-ab)c_{k}+ab(ab-4)d_{k+1}\\
        &=\underbrace{bc_k^2+ad_{k+1}^2-abc_kd_{k+1}-2bc_k+abd_{k+1}+b}+b(ab-4)c_{k+2}
    \end{align*}
    Both the bracketed terms are $>0$ by induction hypothesis and the remaining term is also non-negative since $ab\geq 4.$ Now the proof is completed by Lemma \ref{lemstar}.
\end{proof}

Relations in \ref{rootfourtypes} give us a convenient way to write the sum (and the difference) of two real positive roots as (without loss of generality assume $j\geq k$) 
\begin{equation}\label{altsumanddiffsame}
    \begin{split}
        \beta_{1}^{2k}\pm\beta_{1}^{2j}=&(s_2s_1)^k(\alpha_2\pm\beta_{1}^{2(j-k)}),\  \beta_{1}^{2k+1}\pm\beta_{1}^{2j+1}=s_2(s_1s_2)^k(\alpha_1\pm\beta_{2}^{2(j-k)})\\
        \beta_{2}^{2k}\pm\beta_{2}^{2j}=&(s_1s_2)^k(\alpha_1\pm\beta_{2}^{2(j-k)}),\  \beta_{2}^{2k+1}\pm\beta_{2}^{2j+1}=s_1(s_2s_1)^k(\alpha_2\pm\beta_{1}^{2(j-k)})\\
        \beta_{1}^{2k}\pm\beta_{1}^{2j+1}=&(s_2s_1)^k(\alpha_2\pm\beta_{1}^{2(j-k)+1}),\  \beta_{2}^{2k}\pm\beta_{2}^{2j+1}=(s_1s_2)^k(\alpha_1\pm\beta_{2}^{2(j-k)+1}),
    \end{split}
\end{equation}
Other cases can be checked similarly. In any case, the element $\beta_1^j\pm\beta_2^k$ is $W$ conjugate to either $\alpha_1\pm\beta_2^r$ or $\alpha_2\pm\beta_1^s$ for some $r,s\in\bz_+.$  Similarly, considering the sum and difference of roots of different types, we obtain 
\begin{equation}\label{altdifftype}
    \begin{split}
        \beta_1^{2j}\pm\beta_2^{2k}=(s_1s_2)^{k}&((s_2s_1)^{j+k}\alpha_2\pm\alpha_1),\ \ \beta_1^{2j}\pm\beta_2^{2k+1}=(s_1s_2)^{k}s_1(s_1(s_2s_1)^{j+k}\alpha_2\pm\alpha_2),\\
        &\beta_1^{2k+1}\pm\beta_2^{2j}=(s_2s_1)^{k}s_2(\alpha_1\pm(s_2s_1)^{j+k}s_2\alpha_1),\\
        &\beta_1^{2k+1}\pm\beta_2^{2j+1}=s_2(s_1s_2)^{k}(\alpha_1\pm(s_2s_1)^{j+k+1}\alpha_2).
    \end{split}
\end{equation}
In any case, the element $\beta_1^j\pm\beta_2^k$ is $W$ conjugate to either $\alpha_1\pm\beta_1^r$ or $\alpha_2\pm\beta_2^s$ for some $r,s\in\bz_+.$

\begin{prop}\label{keyproppos}
    Let the roots $\beta_{i}^{k},\ i=1,2$ and $k\in \bz_+$ be given by Equation \ref{realrootstypes}. Then the following hold for $i=1,2.$ If 
    \begin{enumerate}
        \item $b\geq 2,$ then the difference $\beta_{i}^{j}-\beta_{i}^{k}$ is an imaginary root for $j\neq k$ and the difference $\beta_{1}^{j}-\beta_{2}^{k}$ is never a root for any $j,k\in\bz_+.$
        \item $b=1$ and $a\geq 5,$ then $\beta_{i}^{j}-\beta_{i}^{k}$ is an imaginary root when $|j-k|>1.$ Moreover, $\beta_{i}^{j}-\beta_{i}^{j+1}$ is a real root. However, the difference $\beta_{1}^{j}-\beta_{2}^{k}$ is a root if and only if it is a real root if and only if $j=1$ and $k=0.$
    \end{enumerate}
    In particular, $\beta_{i}^{j}-\beta_{i}^{k}$ is always a root for $j\neq k,\ i=1,2.$
\end{prop}
\begin{proof}
We first prove both the statements concerning the difference of the roots of same type namely the case when $\beta_i^j-\beta_i^k$ for $i=1,2$ can be a root. Let $\mu_k=\alpha_2-\beta_{1}^{k}$ and $\nu_k=\alpha_1-\beta_2^k,\ k\in\bz_+.$  It is clear from Equations in \ref{altsumanddiffsame} that it is enough to prove that
    \begin{enumerate}
    \item $\mu_k$ and $\nu_{k}$ are both imaginary roots for $k\geq 1$ if $b\geq 2,$
    \item $\mu_{k}$ and $\nu_{k}$ are both imaginary roots for $k\geq 2$ and a real root for $k=1,$ if $b=1$ and $a\geq 5.$
    \end{enumerate} 
    Then we have 
    $$(\mu_k,\mu_k)=\frac{2}{b}\bigg(bc_k^2+ad_{k+1}^2-abc_kd_{k+1}+abc_{k}-2ad_{k+1}+a\bigg),$$
    $$(\nu_k,\nu_k)=\frac{2}{b}\bigg(bc_{k+1}^2+ad_k^2-abc_{k+1}d_k-2bc_{k+1}+abd_k+b\bigg).$$ By Lemma \ref{lem1} , both $(\mu_k,\mu_k)$ and $(\nu_k,\nu_k)$ are $\le 0$ for all $k\geq 1$ if $b\geq 2$ and for all $k\geq 2$ if $b=1$ and $a\geq 5.$ Rest follows from the fact that $\mu_1$ and $\nu_1$ are both real roots if $b=1$.

    Now we shall consider the case of the difference of two roots of different type. Let $\mu_k=\alpha_1-\beta_{1}^{k},\nu_k=\alpha_2-\beta_{2}^{k},\ k\in\bz_+.$ Using relations in \ref{altdifftype}, it is enough to prove  that $(\mu_k,\mu_k)>0,\ (\nu_k,\nu_k)>0$ for all $k\geq 0$ and 
\begin{enumerate}
    \item $\nu_k$ is never a root,
    \item $\mu_k$ is not a root if $b\geq 2,$
    \item $\mu_k$ is a root if and only if $k=1$ if $a\geq 5,b=1.$
\end{enumerate}
Note that we have 
$$(\mu_k,\mu_k)=\frac{2}{b}\bigg(bc_k^2+ad_{k+1}^2-abc_kd_{k+1}-2bc_k+abd_{k+1}+b\bigg),$$
$$(\nu_k,\nu_k)=\frac{2}{b}\bigg(bc_{k+1}^2+ad_{k}^2-abc_{k+1}d_{k}+abc_{k+1}-2ad_{k}+a\bigg),$$
and both are positive by Lemma \ref{lem2}. If $b\geq 2,$ then $\mu_k$ and $\nu_k$ are not roots by Lemma \ref{sumnotreal}. It remains to prove that if $a\geq 5$ and $b=1,$ then $\nu_k$ is not a root and $\gamma_k$ is a root if and only if $k=1.$ By \cite[Theorem 3.6]{carbone2021commutator}, the alternative description of the real positive roots in \ref{rootfourtypes} and  Lemma \ref{increasing}, among all the roots 
$$(s_2s_1)^{s+t}\alpha_2-\alpha_1,\ (s_1s_2)^{s+t}s_1\alpha_2-\alpha_2,\ \ \alpha_1-(s_2s_1)^{s+t}s_2\alpha_1,\ \ \alpha_1-(s_2s_1)^{s+t+1}\alpha_2$$  only $\alpha_1-(s_2s_1)^{s+t}s_2\alpha_1$ is a real root if and only if $s=t=0.$  In particular, only $\mu_k$ can be a root and it is a root if and only if $k=1.$ This also implies that $\nu_k$ is never a root. This completes the proof.
\end{proof}
\noindent We are now ready to prove Theorem \ref{Thmab}, the main Theorem of this Section.\medskip

\noindent \textbf{Proof of Theorem \ref{Thmab}:}
     Let $\Sigma\subseteq \Delta^+_{\mathrm{re}}$ be a $\pi$-system in $\Delta.$ By Proposition \ref{keyproppos} it can not contain two roots of same type. Hence $\Sigma=\{\beta_1^k,\beta_2^j\}$ for some $k,j\in\bz_+$ and therefore $|\Sigma|\le 2.$ Conversely, if $b>1,$ then by Proposition \ref{keyproppos}, we have that $\Sigma=\{\beta_1^j,\beta_2^k\}$ is a $\pi$-system for any $k,j\in\bz_+.$ If $b=1$ and $a\geq 5,$ then $\Sigma=\{\beta_1^j,\beta_2^k\}$ is a $\pi$-system if and only if $(j,k)\neq (1,0).$ This completes the proof.

\subsection{} In this subsection we shall completely classify (upto isomorphism) the root generated subalgebras of $\lie g$ along with their types using our classification of $\pi$-systems in Theorem \ref{Thmab}. Again, we point out that along with \cite[Theorem 1]{habib2023root}, the generalized Cartan matrix types of root generated subalgebras have been listed in \cite[Table 2]{carbone2021commutator}. In this Section, we determine precisely the types of the root generated subalgebras $\lie g(\Sigma)$ where $\Sigma$ is defined by Theorem \ref{Thmab}. This will give a more explicit description of the root generated subalgebras.
We begin with the following Lemma.
\begin{lem}\label{lem1noncongruent}
    Let the roots $\beta_{i}^{k},\ i=1,2$ and $k\in \bz_+$ be defined by Equation \ref{realrootstypes}. Define two sequences $\eta_j$ and $\gamma_j$ for $j\in\bz_+$ by
    $$\eta_j:=-\langle \beta_1^j,(\beta_2^0)^\vee
    \rangle,\ \ \gamma_j:=-\langle \beta_2^{2j},(\beta_1^0)^\vee
    \rangle.$$
    Then both the sequences $\xi_j:=\eta_{2j},\ \zeta_{j}:=\eta_{2j+1}$ satisfy the recurrence relation $$a_j=(ab-2)a_{j-1}-a_{j-2}$$ with different initial conditions
    $$\xi_0=a,\xi_1=a(ab-3);\ \ \zeta_0=(ab-2),\ \zeta_1= (\zeta_0^2-2).$$ Moreover, we have $a\gamma_j=b\xi_j,$\ $\eta_j$ is constant for $a=b=2$ and 
    $$\begin{cases}
        \eta_1<\eta_0<\eta_3<\eta_2<\eta_5<\cdots\ \ & \text{ if } b=1\text{ and }\ a\geq 5,\\
        \eta_0<\eta_1<\eta_2<\cdots \ \ &\text{ otherwise}.\\
    \end{cases}$$
\end{lem}

\begin{proof}
     It is easy to see that the sequences $\xi_j$ and $\zeta_j$ satisfy the above recurrence since by  \cite[Lemma 2.1]{carbone2021commutator}, both $c_{2j},d_{2j+1},\ j\in\bz_+$ satisfy the same recurrence and $\xi_j=c_{2j+2}-c_{2j},\ \zeta_j=d_{2j+3}-d_{2j+1}$.
    Using the recurrence \ref{recurrenceab}, a simple computation shows that $\eta_{j+1}-\eta_j=c_{j+3}-c_{j+1}+c_j-c_{j+2}.$ Note that comparing with the inequalities in \cite[Lemma 3.1]{carbone2021commutator} we can only conclude that $\eta_{j+1}\geq\eta_j$ for all $j\in\bz_+.$ Instead, again using the recurrences in \ref{recurrenceab}, we obtain
    \begin{equation}\label{dummyeqn}
        \eta_{j+1}-\eta_j=\begin{cases}
        (b-2)c_{2r+2}+(a-2)d_{2r+1}\ & \ \text{ if } j=2r,\\
        (a-2)d_{2r+3}+(b-2)c_{2r+2}\ & \ \text{ if } j=2r+1.
    \end{cases}
    \end{equation}
    Thus $\eta_j$ is constant if $a=b=2.$ If $(a,b)\neq (2,2)$ and $b\neq 1,$ we have $\eta_{j+1}>\eta_j$ for all $j\geq 0.$ For $b=1$ and $a\geq 5,$ using Equation \ref{dummyeqn}, a straightforward computations show that 
    $$\eta_{2k+1}-\eta_{2k-2}\geq d_{2k+1}-d_{2k-1}+d_{2k-3}>0,\  \text{ and }$$
    $$\eta_{2r+1}-\eta_{2r}=d_{2r-1}-d_{2r+1}<0,$$
    for all $k\geq 2$ and $r\geq 1.$ It is trivial to check that $\eta_1<\eta_0<\eta_3$ and therefore $\eta_1<\eta_0<\eta_3<\eta_2<\eta_5<\cdots$ as required. The fact $a\gamma_j=b\xi_j$ follows from an easy computation. This completes the proof.
\end{proof}

The following Lemma is crucial to classify the root generated subalgebras upto isomorphism.
\begin{lem}\label{lem2noncongruent}
    Let the roots $\beta_{i}^{k},\ i=1,2$ and $k\in \bz_+$ be defined by Equation \ref{realrootstypes}. If $j\not\equiv k\ \mathrm{mod}\ 2,$ then the following hold.
    \begin{enumerate}
        \item $\langle\beta_{1}^{j},(\beta_{2}^{k})^\vee\rangle=\langle\beta_{2}^{k},(\beta_{1}^{j})^\vee\rangle$ for all $j,k\in\bz_+.$
        \item if $j+k=\ell+i,$ then $(\beta_{1}^{j},(\beta_{2}^{k})^\vee)=(\beta_{1}^{\ell},(\beta_{2}^{i})^\vee)$.
    \end{enumerate}
\end{lem}
\begin{proof}
$(1)$ follows from the fact $(\beta_1^j,\beta_1^j)=(\beta_2^k,\beta_2^k)$ if $j\not\equiv k\ \mathrm{mod}\ 2$ (c.f. Equation \ref{rootfourtypes}). Note that to prove $(2)$ it is enough to show that $\langle\beta_{1}^{j},(\beta_{2}^{k})^\vee\rangle=\langle\beta_{1}^{j+1},(\beta_{2}^{k-1})^\vee\rangle$ for all $j\in\bz_+$ and $k\geq 1.$ First assume that $j=2s$ and $k=2r+1$ for some $r,s\in\bz_+.$ Using Equation \ref{rootfourtypes}, a direct computations shows that 
$$\langle \beta_{1}^{j},(\beta_{2}^{k})^\vee\rangle=2d_{2(r+s)+1}-bc_{2(r+s)+2},\ \ \langle \beta_{1}^{j+1},\beta_{2}^{k-1}\rangle=2c_{2(r+s)+1}-ad_{2(r+s)+2}.$$
If $j=2s+1$ and $k=2r,$ then again we have 
$$\langle \beta_{1}^{j},(\beta_{2}^{k})^\vee\rangle=2c_{2(r+s)+1}-ad_{2(r+s)+2},\ \ \langle \beta_{1}^{j+1},\beta_{2}^{k-1}\rangle=2d_{2(r+s)+1}-bc_{2(r+s)+2}.$$

Now the result follows from Lemma \ref{lemrelcd}.
\end{proof}

The following easy Lemma deals with the case when $j\equiv k\ \mathrm{mod}\ 2.$ It is easy to prove and has interesting consequences.
\begin{lem}\label{lem1congruent}
    Let the roots $\beta_{i}^{k},\ i=1,2$ and $k\in \bz_+$ be defined by Equation \ref{realrootstypes}. Then the following hold.
    \begin{enumerate}
        \item If $j\equiv k\ \mathrm{mod}\ 2,$ then $(\beta_1^j,\beta_2^k)=(\beta_1^{j+1},\beta_2^{k-1})$ for all $j\in \bz_+$ and $k\geq 1,$
        \item The sequence $\langle \beta_1^{2k},(\beta_2^0)^\vee \rangle$ is decreasing and is strictly decreasing if $ab>4.$
        \item We have $a\langle \beta_2^{2k},(\beta_1^0)^\vee \rangle=b\langle \beta_1^{2k},(\beta_2^0)^\vee \rangle$ for all $k\in\bz_+.$
    \end{enumerate}
\end{lem}
\begin{proof}
    $(1)$ follows from Equation \ref{rootfourtypes} since $j\equiv k\ \mathrm{mod}\ 2.$ Using recurrence \ref{recurrenceab}, an easy computation shows that $\langle \beta_1^{2k},(\beta_2^0)^\vee\rangle = c_{2k}-c_{2k+2}$ and 
    $$\langle \beta_1^{2k},(\beta_2^0)^\vee\rangle-\langle \beta_1^{2k+2},(\beta_2^0)^\vee\rangle=(ab-4)c_{2k+2},$$ and thus $(2)$ is proved. $(3)$ follows from simple computation.
\end{proof}

\begin{rem}\label{remconseq}
    Assume that $j\equiv k\ \mathrm{mod}\ 2.$ The above Lemma has the following consequences.
\begin{enumerate}
    \item $\langle \beta_1^j,(\beta_2^k)^\vee \rangle=\langle \beta_1^{j+2},(\beta_2^{k-2})^\vee \rangle$ and $\langle \beta_2^k,(\beta_1^j)^\vee \rangle=\langle \beta_2^{k+2},(\beta_1^{j-2})^\vee \rangle.$\vspace{0,2cm}\label{1st}
    \item $\langle \beta_1^{j+1},(\beta_2^{k-1})^\vee \rangle=\langle \beta_2^{k},(\beta_1^{j})^\vee \rangle$ and $\langle \beta_2^{k+1},(\beta_1^{j-1})^\vee \rangle=\langle \beta_1^{j},(\beta_2^{k})^\vee \rangle.$\vspace{0,2cm}\label{2nd}
    \item Each of the sequences below is decreasing and strictly decreasing if $ab\geq 5:$
    $$\langle \beta_2^{2r},(\beta_1^0)^\vee\rangle,\ \ \ \langle \beta_1^{2r+1},(\beta_2^1)^\vee\rangle,\ \ \ \langle \beta_2^{2r+1},(\beta_1^1)^\vee\rangle,\ \ r\in\bz_+.$$\label{3rd}
    \item If $a>b,$ then for $r\in\bz_+$ we have
    $$\langle \beta_2^{2r+1},(\beta_1^1)^\vee\rangle<\langle \beta_1^{2r+1},(\beta_2^1)^\vee\rangle<-1,\ \ \langle \beta_1^{2r},(\beta_2^0)^\vee\rangle<\langle \beta_2^{2r},(\beta_1^0)^\vee\rangle\le -1,$$ where the equality holds if and only if $r=0$ and $b=1$ (for any $a$).\vspace{0,2cm}\label{4th}
    \item If $a=b,$ then $\langle \beta_1^{2r+1},(\beta_2^1)^\vee\rangle=\langle \beta_2^{2r+1},(\beta_1^1)^\vee\rangle.$ Moreover, if $a\neq 2,$ then $$\langle \beta_2^{2r},(\beta_1^0)^\vee\rangle >\langle \beta_1^{2r+1},(\beta_2^1)^\vee\rangle.$$ Otherwise i.e. when $a=b=2,$ we have $\langle \beta_2^{2r},(\beta_1^0)^\vee\rangle=\langle \beta_1^{2r+1},(\beta_2^1)^\vee\rangle$ for all $r\in\bz_+.$ \label{5th}
\end{enumerate}
\end{rem}

\noindent The following Theorem classifies the root generated subalgebras of a rank $2$ Kac-Moody algebra upto isomorphism.
\begin{thm}
    Let $A=\begin{psmallmatrix}
        2& -a\\ -b&2
    \end{psmallmatrix}$ be a GCM of rank $2$ such that either $b\geq 2$ or $b=1$ and $a\geq 5.$ Let $\Sigma_{j,k}:=\{\beta_1^j,\beta_2^k\}$ be the $\pi$-system defined in Theorem \ref{Thmab}. Let $B_{\Sigma_{j,k}}$ be the matrix associated to the $\pi$-system $\Sigma_{j,k}.$
    \begin{enumerate}
        \item If $j\not\equiv k\ \mathrm{mod}\ 2,$ then the matrix $B_{\Sigma_{j,k}}$ is symmetric and we have $\lie g(\Sigma_{j,k})\cong \lie g(\Sigma_{j+k,0})$ whenever $\Sigma_{j+k,0}$ is a $\pi$-system. Moreover, \vspace{0,2cm}
        \begin{enumerate}
            \item if $ab\geq 5,$ then $\lie g(\Sigma_{2r+1,0})\cong \lie g(\Sigma_{2s+1,0})$ if and only if $r=s,$\vspace{0,2cm}
            \item if $a=b=2,$ then $\lie g(\Sigma_{2r+1,0})\cong\lie g(\Sigma_{2s+1,0})$ for all $r,s\in\bz_+.$
        \end{enumerate}
        \medskip
        
        \noindent Conversely, given a symmetric generalized Cartan matrix $S=\begin{psmallmatrix}
            2& -s\\ -s& 2
        \end{psmallmatrix},$ we have $S=B_{\Sigma_{j,k}}$ for some $\pi$-system $\Sigma_{j,k}$ defined above such that $j\not\equiv k\ \mathrm{mod}\ 2$ if and only if $s$ is a solution of the recurrence relation
        $$a_n=(ab-2)a_{n-1}-a_{n-2},\ \ \ a_0=ab-2,\ a_1=a_0^2-2.$$
        
        \item If $j\equiv k\ \mathrm{mod}\ 2,$ then the matrix $B_{\Sigma_{j,k}}$ is symmetric if and only if $a=b$ and we have $$B_{\Sigma_{j,k}}= \begin{cases}
            B_{\Sigma_{j+k,0}}\ &\text{ if } j\equiv 0\ \mathrm{mod}\ 2,\\
            B_{\Sigma_{j+k-1,1}}\ &\text{ if }j\not\equiv 0\ \mathrm{mod}\ 2.
        \end{cases}$$
        Moreover, 
        \begin{enumerate}
            \item we have $\lie g(\Sigma_{2r,0})\cong \lie g(\Sigma_{2r-1,1})$ for all $r\geq 1,$\vspace{0,2cm}
            \item if $ab\geq 5,$ then $\lie g(\Sigma_{2r,0})\cong \lie g(\Sigma_{2s,0})$ if and only if $r=s,\ \ r,s\in\bz_+,$\vspace{0,2cm}
            \item if $a=b=2,$ then $\lie g(\Sigma_{2r,0})\cong \lie g(\Sigma_{2s,0})$ for all $r,s\in \bz_+.$
        \end{enumerate}
        \medskip
        
        \noindent Conversely, given a generalized Cartan matrix $C=\begin{psmallmatrix}
            2& -c_1\\ -c_2& 2
        \end{psmallmatrix},$ we have $C=B_{\Sigma_{j,k}}$ for some 
        $j,k\in\bz_+$ with $j\equiv k\ \mathrm{mod}\ 2$ if and only if $(c_1,c_2)=(x_n,y_n)$ for some $n\in\bz_+$ where $x_n$ and $y_n$ are given by the recurrences  
        $$x_n=(ab-2)x_{n-1}-x_{n-2},\ \ \ x_0=b,\ x_1=b(ab-3),$$
        $$y_n=(ab-2)y_{n-1}-y_{n-2},\ \ \ y_0=a,\ y_1=a(ab-3),$$
    \end{enumerate}
\end{thm}
\begin{proof}
    The statements in $(1)$ follow from Lemma \ref{lem1noncongruent} and Lemma \ref{lem2noncongruent}. To prove $(2),$ first note that $B_{\Sigma_{2k,0}}=B_{\Sigma_{2k-1,1}}^t$ and thus proving $(a).$ The conclusion in $(c)$ follows from Remark \ref{remconseq}. The only case remains to prove is $2(b).$ Since $\eta_{2j},\ j\in\bz_+$ is strictly increasing by Lemma \ref{lem1noncongruent}, we have $\lie g(\Sigma_{2k,0})\cong \lie g(\Sigma_{2r,0})$ for some $k<r$ if and only if $B_{\Sigma_{2k,0}}=B_{\Sigma_{2r,0}}^t.$ Note that $B_{\Sigma_{2k,0}}=B_{\Sigma_{2r,0}}^t$ if and only if \begin{align*}
    &\quad \langle\beta_1^{2r},(\beta_2^{0})^\vee\rangle=\langle\beta_2^0,(\beta_1^{2k})^\vee\rangle \text{ and  } \langle\beta_1^{2k},(\beta_2^{0})^\vee\rangle=\langle\beta_2^0,(\beta_1^{2r})^\vee\rangle,\\
    \Longleftrightarrow &\quad c_{2r}-c_{2r+2}=d_{2k}-d_{2k+2} \text{ and } c_{2k}-c_{2k+2}=d_{2r}-d_{2r+2}.
    \end{align*}
    Since $k<r,$ using Lemma \ref{lemrelcd} and Remark \ref{1st}, we obtain 
    \begin{align*}
        &\frac{b}{a}(c_{2r}-c_{2r+2})=d_{2r}-d_{2r+2}=\langle\beta_2^0,(\beta_1^{2r})^\vee\rangle=\langle\beta_2^{2r},(\beta_1^{0})^\vee\rangle\\
        &<\langle\beta_2^{2k},(\beta_1^{0})^\vee\rangle=\langle\beta_2^0,(\beta_1^{2k})^\vee\rangle=-d_{2k+2}+d_{2k}=c_{2r}-c_{2r+2},
    \end{align*}
    which in turn implies $\frac{(b-a)}{a}(c_{2r}-c_{2r+2})<0$
    which is impossible by Lemma \ref{increasing}. This completes the proof.
\end{proof}

Table \ref{tablesummary} summarizes all the results obtained in this Section, the sequences $\xi_j$ and $\zeta_j$ are given by Lemma \ref{lem1noncongruent}.

\medskip

\begin{table}[ht]
    \centering
    \renewcommand{\arraystretch}{2}
    \begin{tabular}{|c|c|c|c|c|c|}
 \hline
 Case  & $\pi$-system $\Sigma$ of the form & Subcase   & Nature of $B_{\Sigma}$   & Form of $B_\Sigma$    & Type   \\
 \hline\hline
 \multirow{3}*{$ab\geq 5$} & $\Sigma_{2j+1,0}$ & any &symmetric  &  $\begin{psmallmatrix}
     2 & -\zeta_{j}\\ -\zeta_{j}&2
 \end{psmallmatrix}$   & Hyperbolic    \\
 \cline{2-6} &  \multirow{2}*{$\Sigma_{2j,0}$}   & $a= b$ & symmetric  & $\begin{psmallmatrix}
     2 & -\xi_{j}\\ -\xi_{j}&2
 \end{psmallmatrix}$   & Hyperbolic \\
\cline{3-6}
 & & $a\neq b$ & non-symmetric  & $\begin{psmallmatrix}
     2 & -\frac{b}{a}\xi_{j}\\ -\xi_{j}&2
 \end{psmallmatrix}$  & Hyperbolic \\
 \hline
 \multirow{2}*{$a=b=2$} & $\Sigma_{1,0}$ & any  & symmetric  & $\begin{psmallmatrix}
     2 & -2\\ -2 &2
 \end{psmallmatrix}$    & Affine   \\
 \cline{2-6} &  $\Sigma_{0,0}$   & any  & symmetric   & $\begin{psmallmatrix}
     2 & -2\\ -2&2
 \end{psmallmatrix}$& Affine \\
\hline
\end{tabular}
    \caption{Root generated subalgebras of $\lie g(A)$}
    \label{tablesummary}
\end{table}

\section{Classifications of \texorpdfstring{$\pi$}{pi}-systems} \label{clpisystems}
In this Section we aim to classify all the $\pi$-systems of a rank $2$ Kac-Moody algebra. From Theorem \ref{Thmab} it is clear that a $\pi$-system $\Sigma\subseteq \Delta^{\mathrm{re}}$ satisfies $|\Sigma|\le 4.$ We shall systematically investigate the possibilities. Eventually we shall show that any $\pi$-system $\Sigma$ satisfies $|\Sigma|\le 2.$ To understand the $\pi$-system including negative roots, we need to understand when the sum of two positive real roots is again a root. We answer this question in what follows in this Section. In \cite{carbone2021commutator} the authors described explicitly when a sum of two real roots is a real root. The main result of this Section is the following.
\begin{thm}\label{NThmab}
    Let $A= \begin{psmallmatrix} 2 & -a \\-b & 2 \end{psmallmatrix}$ be a generalized Cartan matrix. Then every $\pi$-system $\Sigma\subseteq\Delta_{\mathrm{re}}$ of $\lie g(A)$ such $|\Sigma|>1$ is either given by Theorem \ref{Thmab} (or it's negative) or is of the form $\Sigma=\{\beta_i^j,-\beta_i^k\}$ for some $j,k\in\bz_+$ and $i=1,2.$ In particular, $\Sigma$ satisfies $|\Sigma|\le 2$ and therefore $\Sigma$ is linearly independent unless $\Sigma=\{\alpha,-\alpha\}$ for some $\alpha\in\Delta^+_{\mathrm{re}}.$ 
    
    Conversely, if $b\geq 2,$ then the subset $\Sigma=\{\beta_i^j,-\beta_i^k\}$ is a $\pi$-system for any $j,k\in \bz_+$ and $i=1,2.$ If $b=1$ and $a\geq 5,$ then $\Sigma=\{\beta_i^j,-\beta_i^k\}$ is a $\pi$-system if and only if $\{j,k\}\neq \{s,s+2\}$ for some $s\in\bz_+$ such that $s$ is even (resp. odd)  if $i$ is even (resp. odd).
\end{thm}

\subsection{} We begin this subsection with a couple of Lemmas which provide important inequalities to prove Theorem \ref{NThmab}.
\begin{lem}\label{lem41}
    For all $k\in\bz_+,$ the following strict inequalities hold:
    \begin{equation}\label{eq23}
        bc_k^2+ad_{k+1}^2-abc_kd_{k+1}-abc_k+2ad_{k+1}+a> 0.
    \end{equation}
    \begin{equation}\label{eq33}
        bc_{k+1}^2+ad_{k}^2-abc_{k+1}d_{k}+2bc_{k+1}-abd_{k}+b> 0.
    \end{equation}
\end{lem}
\begin{proof}
    For $k=0$ the above equations become $4a$ and $4b$ respectively and for $k=1,$ both equations become $a+b+ab$ which are all strictly positive. Now we have 
    \begin{align*}
        &bc_{k+2}^2+ad_{k+3}^2-abc_{k+2}d_{k+3}-abc_{k+2}+2ad_{k+3}+a\\&=        \underbrace{bc_k^2+ad_{k+1}^2-abc_kd_{k+1}-abc_k+2ad_{k+1}+a}+a(ab-4)d_{k+1}
    \end{align*}
    and
    \begin{align*}
        & bc_{k+3}^2+ad_{k+2}^2-abc_{k+3}d_{k+2}+2bc_{k+3}-abd_{k+2}+b\\&=        \underbrace{bc_{k+1}^2+ad_{k}^2-abc_{k+1}d_{k}+2bc_{k+1}-abd_{k}+b}+b(ab-4)c_{k+1}
    \end{align*}
    By induction the bracketed terms are positive, and the remaining terms are non-negative since $ab\geq 4$. Now the result follows from Lemma \ref{lemstar}.
\end{proof}
\begin{rem}
    Note that these inequalities are independent of those obtained in Lemma \ref{lem2}. For example, for $a=5,b=1$ Equation \ref{eq23} stays grater than Equation \ref{funnyeq01} in Lemma \ref{lem2} for $k=0,1,\cdots,5$ whereas Equation \ref{eq23} alternates between being greater and lesser than Equation \ref{funnyeq00} in Lemma \ref{lem2}. A similar pattern applies to Equation \ref{eq33}.
\end{rem}

\begin{lem}\label{lem42}
    Both of the inequalities hold:
    \begin{equation}\label{eq24}
    bc_{k+1}^2+ad_k^2-abc_{k+1}d_k+2ad_k-abc_{k+1}+a\leq 0
    \end{equation}
    \begin{equation}\label{eq34}
    bc_k^2+ad_{k+1}^2-abd_{k+1}c_k-abd_{k+1}+2bc_k+b\leq 0
    \end{equation}
    \begin{enumerate}
        \item for all $k\geq 0$ if $b\geq2$.
        \item for all $k\geq 1$ if $b=1$ and $a\geq 5$.
    \end{enumerate}
\end{lem}

\begin{proof}
    For $k=1$ the above expressions become $a(4-ab)$ and $b(4-ab)$ respectively and for $k=2$, they both become $a+b-ab(ab-3)$. It is easy to see that they are non-positive by the given conditions. Again using the recurrence \ref{recurrenceab} and Equation \ref{auxeq} we obtain
    \begin{align*}
        &bc_{k+3}^2+ad_{k+2}^2-abc_{k+3}d_{k+2}+2ad_{k+2}-abc_{k+3}+a\\
        &=bc_{k+1}^2+ad_k^2-abc_{k+1}d_k+2ad_k-abc_{k+1}+a-a(4-ab)d_{k}+ab(4-ab)c_{k+1}\\
        &=\underbrace{bc_{k+1}^2+ad_k^2-abc_{k+1}d_k+2ad_k-abc_{k+1}+a}+a(4-ab)d_{k+2}
    \end{align*}
    and
    \begin{align*}
        &bc_{k+2}^2+ad_{k+3}^2-abd_{k+3}c_{k+2}-abd_{k+3}+2bc_{k+2}+b\\
        &=bc_k^2+ad_{k+1}^2-abd_{k+1}c_k-abd_{k+1}+2bc_k+b-b(4-ab)c_{k}+ab(4-ab)d_{k+1}\\
        &=\underbrace{bc_k^2+ad_{k+1}^2-abd_{k+1}c_k-abd_{k+1}+2bc_k+b}+b(4-ab)c_{k+2}
    \end{align*}
    By induction the bracketed terms are less than or equal to zero and the remaining terms are less than or equal to zero since $ab\geq 4$. The rest follows from Lemma \ref{lemstar} since Equation \ref{eq24} and \ref{eq34} hold for $k=0$ if $b\geq 2.$
\end{proof}
\begin{rem}
    Similarly, the inequalities obtained above are also independent of those obtained in Lemma \ref{lem1}.
\end{rem}

The following Proposition determines when the sum of two real roots is again a root.
\begin{prop}\label{Nsumroot}
Let the roots $\beta_i^k,\ i=1,2,\ k\in\bz_+$ be defined by Equation \ref{realrootstypes}. Then 
\begin{enumerate}
    \item the sum $\gamma:=\beta_{i}^{k}+\beta_{i}^{j}$ satisfies $(\gamma,\gamma)>0$ for $i=1,2.$ Moreover, if 
    \begin{enumerate}
        \item $b\geq 2,$ then the sum $\beta_{i}^{k}+\beta_{i}^{j},\ i=1,2$ is never a root,
        \item $b=1$ and $a\geq 5,$ then $\beta_{1}^{2k+1}+\beta_{1}^{2k+3}$ and $\beta_{2}^{2k}+\beta_{2}^{2k+2},\ k\in\bz_+$ are all real roots and they are the only pair of roots of same type whose sum is a real root.
    \end{enumerate}
    \item the sum $\gamma:=\beta_{1}^{k}+\beta_{2}^{j}$ is always a root. Moreover, $\gamma$ satisfies $(\gamma,\gamma)\le 0$ unless $b=1$ and $(k,j)=(0,0).$ 
\end{enumerate}
\end{prop}
\begin{proof}
\begin{enumerate}
    \item It is clear from Equations in \ref{altsumanddiffsame} that it is enough to prove 
\begin{enumerate}
    \item the sums $\alpha_1+\beta_2^k$ and $\alpha_2+\beta_1^k$ have positive norm for any $k\in\bz_+,$
    \item if $b\geq 2,$ then $\alpha_1+\beta_2^r$ and $\alpha_2+\beta_1^r$ are not roots for all $r\geq 1,$ 
    \item if $b=1$ and $a\geq 5,$  then $\alpha_2+\beta_1^r$ is never a root for $r\geq 1$ and $\alpha_1+\beta_2^r$ is a root if and only if $r=2.$
\end{enumerate}
 If we define $\mu_k=\alpha_1+\beta_2^k$ and $\nu_k=\alpha_2+\beta_1^k,$ then simple computations show that 
 $$(\mu_k,\mu_k)=\frac{2}{b}(bc_{k+1}^2+ad_{k}^2-abc_{k+1}d_{k}+2bc_{k+1}-abd_{k}+b),$$ 
 $$(\nu_k,\nu_k)=\frac{2}{b}(bc_k^2+ad_{k+1}^2-abc_kd_{k+1}-abc_k+2ad_{k+1}+a).$$
Therefore $(\mu_k,\mu_k)>0$ and $(\nu_k,\nu_k)>0$ for all $k\geq 0$ by Lemma \ref{lem41}. Now the case $b\geq 2$ follows from Lemma \ref{sumnotreal}.
 We assume for the rest of the proof that $a\geq 5$ and $b=1.$ By \cite[Theorem 3.6]{carbone2021commutator}, Equation \ref{rootfourtypes} and Lemma \ref{increasing}, the sum $\beta_2^k+\alpha_1$ is a real root if and only if $\beta_2^k=\alpha_2$ or $\beta_2^k=(a-1)\alpha_1+\alpha_2$. The case $\beta_2^k=\alpha_2$ is impossible since $\alpha_2$ is a root of type $I.$ If $\beta_2^k=(a-1)\alpha_1+\alpha_2,$ then Lemma \ref{increasing} guarantees $k=2$ and therefore we obtain the required result. The argument for $\alpha_2+\beta_1^k$ is similar.
 \item Again by Equations in \ref{altdifftype} it is enough to prove that
\begin{enumerate}
    \item  the sums $\alpha_1+\beta_1^k$ and $\alpha_2+\beta_2^k$ have non-positive norms for all $k\geq 0$ if $b\geq 2$,
    \item if $b=1$ and $a\geq 5,$ then 
    \begin{enumerate}
        \item $\alpha_2+\beta_2^k$ has non-positive norm for each $k \geq 1,$
        \item $\alpha_1+\beta_1^k$ has non-positive norm for all $k\geq 1$ and for $k=0,$ it is a real root.
    \end{enumerate}
    \end{enumerate}
    If we define $\mu_k=\alpha_1+\beta_1^k$ and $\nu_k=\alpha_2+\beta_2^k$, then it is easy to see that 
    $$(\mu_k,\mu_k)=\frac{2}{b}(bc_k^2+ad_{k+1}^2-abd_{k+1}c_k-abd_{k+1}+2bc_k+b),$$
    $$(\nu_k,\nu_k)=\frac{2}{b}(bc_{k+1}^2+ad_k^2-abc_{k+1}d_k+2ad_k-abc_{k+1}+a).$$
    Using Lemma \ref{keylem} and Lemma \ref{lem42}, we obtain that $\mu_k$ and $\nu_k$ are imaginary roots for $k\geq 1$ (resp. $k\geq 0$) if $a\geq 5$ and $b=1$ (resp. $b\geq 2$). In the former case, the root $\beta_1^0+\beta_2^0$ is real. This proves the Proposition.
\end{enumerate}
\end{proof}

We are now ready to prove the main Theorem of this Section, namely Theorem \ref{NThmab}.
\medskip

\noindent \textbf{Proof of Theorem \ref{NThmab}}: Let $\Sigma\subseteq \Delta_{\mathrm{re}}$ be a $\pi$-system in $\Delta.$ If either $\Sigma\subseteq \Delta^+$ or $\Sigma\subseteq \Delta^-,$ then the result follows from Theorem \ref{Thmab}. If $\{\alpha,-\beta\}\subseteq\Sigma$ for some $\alpha,\beta\in \Delta^+_{\mathrm{re}},$ then Proposition \ref{Nsumroot} guarantees that $\alpha=\beta_i^{k},\beta=\beta_i^j$ for some $j,k\in\bz_+$ and $i\in\{1,2\}.$ Using Theorem \ref{Thmab} and Proposition \ref{Nsumroot} once more, we obtain $\Sigma=\{\alpha,-\beta\}.$ Moreover, if $b=1$ and $a\geq 5,$ then by Proposition \ref{Nsumroot}, we must have $\{j,k\}\neq \{s,s+2\}$ for some $s\in\bz_+$ such that $s$ is even (resp. odd)  if $i$ is even (resp. odd). Converse follows from Proposition \ref{Nsumroot}.\qed

\subsection{}\label{borcherds} For this particular subsection, we shall allow imaginary roots in $\pi$-systems i.e. a $\pi$-system $\Sigma$ is by definition a subset of $\Delta^+$ such that $\alpha-\beta\notin\Delta$ for all $\alpha,\beta\in \Sigma.$ This definition is originally by Naito in \cite{naito1992regular} where he showed that even Lie algebras of Borcherds Kac-Moody types can be embedded inside a Kac-Moody algebra. We will illustrate through examples that, in this scenario, the classification of $\pi$-systems may be a challenging and intricate problem. Let $A$ be the rank $2$ GCM defined by $A_a:=\begin{psmallmatrix}
    2& -a\\ -3& 2
\end{psmallmatrix}$ such that $a\geq 3.$ Define two subsets $\Sigma_1,\Sigma_2\subseteq \Delta^+$ by $$\Sigma_1:=\left\{(2k+1)\alpha_1+\alpha_2:k=1,2,\cdots,\bigg\lfloor\frac{a-2}{2}\bigg\rfloor\right\},$$
$$\Sigma_2:=\{\alpha_2,\alpha_2+a\alpha_1\}\cup\left\{2i\alpha_1+\alpha_2:i=1,2,\cdots,\bigg\lfloor \frac{a}{2}\bigg\rfloor-1\right\}.$$

A straightforward calculation shows that $\Sigma_1,\Sigma_2\backslash\{\alpha_2,\alpha_2+a\alpha_1\}\subseteq \Delta^+_{\mathrm{im}}$ and both $\Sigma_1,\Sigma_2$ are $\pi$-systems. Thus given any positive integer $r\in\bz$ there exists $a\in\bz_+$ and $\pi$-system $\Sigma^{r}$ in $\lie g(A_a)$ such that $|\Sigma^{r}|>r.$ Now if $\Pi=\{\alpha,\beta\}$ is a subset of either $\Sigma_1$ or $\Sigma_2$ such that $|\Pi|\le 2,$ then $\Pi$ is linearly independent and by \cite{naito1992regular}, the subalgebra $\lie g(\Pi)$ of $\lie g$ generated by $x_{\alpha}^\pm,x_{\beta}^\pm$ is a Borcherds Kac-Moody algebra, where for $x_{\alpha}^\pm$ (resp. $x_{\beta}^\pm$) is any non-zero root vector in $\lie g_{\pm \alpha}$ (resp. $\lie g_{\pm\beta}$).

\section{Appendix}\label{app}
In this last Section we shall consider the remaining case namely $a=4$ and $b=1.$ In this case the Lie algebra $\lie g(A)$ is the affine Lie algebra isomorphic to $BC_{2}^{(2)}.$ We shall prove that similar results to Theorem \ref{Thmab} and Theorem \ref{NThmab} hold in this case as well. The set of real roots of $\lie g$ is given by \cite{kac1990infinite}
\begin{equation}\label{realrootreduced}
    \Delta=\{\pm\epsilon+r\delta:r\in\bz\}\cup\{\pm2\epsilon+(2r+1)\delta:r\in\bz\},
\end{equation}
 where $\delta$ is the indivisible imaginary root, $\epsilon\in\bc$ satisfies $(\epsilon,\epsilon)=1$ and $(\epsilon,\delta)=0.$ A basis of $\Delta$ is given by $\alpha_1=\epsilon,\ \alpha_2=-2\epsilon+\delta.$ A simple computation shows that $$(s_2s_1)^k(\epsilon)=\epsilon-k\delta,\ \ (s_1s_2)^k(\epsilon)=\epsilon+k\delta.$$
Following the notation of \ref{rootfourtypes}, we obtain
$$\beta_1^{2k}=-2\epsilon+(2k+1)\delta,\ \ \ \beta_1^{2k+1}=-\epsilon+(k+1)\delta,$$
$$\beta_2^{2k}=\epsilon+k\delta,\ \ \ \beta_2^{2k+1}=2\epsilon+(2k+1)\delta.$$

From the description of the roots above we have for $i=1,2$
\begin{enumerate}
    \item $\beta_i^{j}-\beta_i^{k}$ is real if and only if $j\not\equiv\ k\ \mathrm{mod}\ 2,$
    \item $\beta_i^{j}-\beta_i^{k}$ is imaginary if and only if $j\equiv\ k\ \mathrm{mod}\ 2$
\end{enumerate}
In particular, $\beta_i^{j}-\beta_i^{k}$ is always a root for $j\neq k\in\bz_+,$ $i=1,2.$ 

Similarly, for the roots of different types, we have that the norm of $\beta_1^j-\beta_2^k$ is always positive. Using Equation \ref{realrootreduced}, we obtain that $\beta_1^j-\beta_2^k$ is a root if and only if $j\equiv 1\ \mathrm{mod}\ 2$ and $k\equiv j +3\ \mathrm{mod}\ 4.$ 

Similarly, for the sum of real roots we obtain that 
\begin{enumerate}
    \item For $i=1,2,$ the sum $\beta_i^j+\beta_i^k$ is a root if and only if it is a real root if and only if $$j\equiv i\ \mathrm{mod}\ 2, k\equiv i\ \mathrm{mod}\ 2\text{ and } j\not\equiv k\ \mathrm{mod}\ 4.$$
    \item The sum $\beta_1^j+\beta_2^k$ is always a root for any $j,k\in\bz_+.$ Moreover, $\beta_1^j+\beta_2^k$ is real (resp. imaginary) if and only if $j\equiv k\ \mathrm{mod}\ 2$ (resp. $j\not\equiv k\ \mathrm{mod}\ 2$).
\end{enumerate}

Thus we obtain the classification of $\pi$-systems in this case as follows.
\begin{thm}\label{thm41}
    Let $A=\begin{psmallmatrix}
        2& -4\\ -1 &2
    \end{psmallmatrix}$ be the affine GCM and $\lie g$ be the corresponding affine Kac-Moody algebra. Then 
    \begin{enumerate}
        \item every $\pi$-system $\Sigma\subseteq \Delta^+_{\mathrm{re}}$ of $\lie g$ such that $|\Sigma|>1$ is of the form $\Sigma=\{\beta_1^j,\beta_2^k\}$ for some $j,k\in\bz_+.$ Conversely, $\Sigma=\{\beta_1^j,\beta_2^k\}$ is a $\pi$-system if and only if either $j\not\equiv 1\ \mathrm{mod}\ 2$ or $k\not\equiv j+3\ \mathrm{mod}\ 4.$
        \item every $\pi$-system $\Sigma\subseteq \Delta_{\mathrm{re}}$ such that $|\Sigma|>1$ of $\lie g(A)$ is either given by $(1)$ or is of the form $\Sigma=\{\beta_i^j,-\beta_i^k\}$ for some $j,k\in\bz_+$ and $i=1,2.$ Conversely, the subset $\Sigma=\{\beta_i^j,-\beta_i^k\}$ is a $\pi$-system if and only if at least one of the following conditions holds.
        $$j\not\equiv i\ \mathrm{mod}\ 2,\ \ \ k\not\equiv i\ \mathrm{mod}\ 2,\ \ \ j\equiv k\ \mathrm{mod}\ 4.$$
    \end{enumerate}
\end{thm}
\begin{rem}
    Note that for any $\pi$-system $\Sigma$ given by Theorem \ref{thm41}, the corresponding matrix $B_{\Sigma}$ is always of affine type.
\end{rem}

\bibliographystyle{plain}
\bibliography{bibliography}

\begin{thebibliography}{10}

\bibitem{anderson10KacMoody}
Kasper K.~S. Andersen, Lisa Carbone, and Diego Penta.
\newblock Kac-{M}oody {F}ibonacci sequences, hyperbolic golden ratios, and real quadratic fields.
\newblock {\em J. Comb. Number Theory}, 2(3):245--278, 2010.

\bibitem{barnea1998graded}
Y.~Barnea, A.~Shalev, and E.~I. Zelmanov.
\newblock Graded subalgebras of affine {K}ac-{M}oody algebras.
\newblock {\em Israel J. Math.}, 104:321--334, 1998.

\bibitem{BHV23a}
Dipnit Biswas, Irfan Habib, and R.~Venkatesh.
\newblock On symmetric closed subsets of real affine root systems.
\newblock {\em J. Algebra}, 628:212--240, 2023.

\bibitem{BdS}
A.~Borel and J.~De~Siebenthal.
\newblock Les sous-groupes ferm\'es de rang maximum des groupes de {L}ie.
\newblock {\em Comment. Math. Helv.}, 23:200--221, 1949.

\bibitem{carbone2015root}
Lisa Carbone, Matt Kownacki, Scott~H Murray, and Sowmya Srinivasan.
\newblock Root subsystems of rank 2 hyperbolic root systems.
\newblock {\em arXiv preprint arXiv:1506.05405}, 2015.

\bibitem{carbone2021commutator}
Lisa Carbone, Matt Kownacki, Scott~H. Murray, and Sowmya Srinivasan.
\newblock Commutator relations and structure constants for rank 2 {K}ac-{M}oody algebras.
\newblock {\em J. Algebra}, 566:443--476, 2021.

\bibitem{carbone2021varvec}
Lisa Carbone, K.~N. Raghavan, Biswajit Ransingh, Krishanu Roy, and Sankaran Viswanath.
\newblock {$\pi$}-systems of symmetrizable {K}ac-{M}oody algebras.
\newblock {\em Lett. Math. Phys.}, 111(1):Paper No. 5, 24, 2021.

\bibitem{damour03cosmological}
T.~Damour, M.~Henneaux, and H.~Nicolai.
\newblock Cosmological billiards.
\newblock {\em Classical Quantum Gravity}, 20(9):R145--R200, 2003.

\bibitem{damour09fermionic}
Thibault Damour and Christian Hillmann.
\newblock Fermionic {K}ac-{M}oody billiards and supergravity.
\newblock {\em J. High Energy Phys.}, 2009(8):100, 55, 2009.

\bibitem{DdeG21}
Andrew Douglas and Willem~A. de~Graaf.
\newblock Closed subsets of root systems and regular subalgebras.
\newblock {\em J. Algebra}, 565:531--547, 2021.

\bibitem{dynkin1952semisimple}
E.~B. Dynkin.
\newblock Semisimple subalgebras of semisimple {L}ie algebras.
\newblock {\em Mat. Sbornik N.S.}, pages 349--462 (3 plates), 1952.

\bibitem{feingold83ahyperbolic}
Alex~J. Feingold and Igor~B. Frenkel.
\newblock A hyperbolic {K}ac-{M}oody algebra and the theory of {S}iegel modular forms of genus {$2$}.
\newblock {\em Math. Ann.}, 263(1):87--144, 1983.

\bibitem{feingold2004subalgebras}
Alex~J. Feingold and Hermann Nicolai.
\newblock Subalgebras of hyperbolic {K}ac-{M}oody algebras.
\newblock In {\em Kac-{M}oody {L}ie algebras and related topics}, volume 343 of {\em Contemp. Math.}, pages 97--114. Amer. Math. Soc., Providence, RI, 2004.

\bibitem{felikson2008regular}
Anna Felikson, Alexander Retakh, and Pavel Tumarkin.
\newblock Regular subalgebras of affine {K}ac-{M}oody algebras.
\newblock {\em J. Phys. A}, 41(36):365204, 16, 2008.

\bibitem{habib2023root}
Irfan Habib, Deniz Kus, and R~Venkatesh.
\newblock Root generated subalgebras of symmetrizable kac-moody algebras.
\newblock {\em arXiv preprint arXiv:2311.12583}, 2023.

\bibitem{superspace81}
S.~W. Hawking and M.~Ro\v{c}ek, editors.
\newblock {\em Superspace and supergravity}. Cambridge University Press, Cambridge-New York, 1981.

\bibitem{Henn06Geometric}
Marc Henneaux, Mauricio Leston, Daniel Persson, and Philippe Spindel.
\newblock Geometric configurations, regular subalgebras of {$E_{10}$} and {M}-theory cosmology.
\newblock {\em J. High Energy Phys.}, 2006(10):021, 48, 2006.

\bibitem{kac1990infinite}
Victor~G. Kac.
\newblock {\em Infinite-dimensional {L}ie algebras}.
\newblock Cambridge University Press, Cambridge, third edition, 1990.

\bibitem{KV21a}
Deniz Kus and R.~Venkatesh.
\newblock Borel--de {S}iebenthal theory for affine reflection systems.
\newblock {\em Mosc. Math. J.}, 21(1):99--127, 2021.

\bibitem{morita1989certain}
Jun Morita.
\newblock Certain rank two subsystems of {K}ac-{M}oody root systems.
\newblock In {\em Infinite-dimensional {L}ie algebras and groups ({L}uminy-{M}arseille, 1988)}, volume~7 of {\em Adv. Ser. Math. Phys.}, pages 52--56. World Sci. Publ., Teaneck, NJ, 1989.

\bibitem{naito1992regular}
Satoshi Naito.
\newblock On regular subalgebras of {K}ac-{M}oody algebras and their associated invariant forms. {S}ymmetrizable case.
\newblock {\em J. Math. Soc. Japan}, 44(2):157--177, 1992.

\bibitem{roy2019maximal}
Krishanu Roy and R.~Venkatesh.
\newblock Maximal closed subroot systems of real affine root systems.
\newblock {\em Transform. Groups}, 24(4):1261--1308, 2019.

\bibitem{sagemath}
{The Sage Developers}.
\newblock {\em {S}ageMath, the {S}age {M}athematics {S}oftware {S}ystem ({V}ersion 9.3)}, 2021.
\newblock {\tt https://www.sagemath.org}.

\bibitem{viswanath08embedding}
Sankaran Viswanath.
\newblock Embeddings of hyperbolic {K}ac-{M}oody algebras into {$E_{10}$}.
\newblock {\em Lett. Math. Phys.}, 83(2):139--148, 2008.

\end{thebibliography}
\end{document}